\documentclass[a4paper,12pt]{article}

\usepackage{amssymb}
\usepackage{epsfig}
\usepackage{amsfonts}
\usepackage{amsmath}
\usepackage{euscript}
\usepackage{amscd}
\usepackage{amsthm}
\usepackage{theoremref}
\usepackage{enumerate}
\usepackage{array}
\usepackage{mathrsfs}
\usepackage{tikz-cd}
\usepackage{float}

\usepackage{abstract}

\usepackage{lipsum}

\DeclareMathAlphabet{\mathpzc}{OT1}{pzc}{m}{it}
\DeclareMathAlphabet{\mathpzc}{OT1}{pzc}{m}{it}

\newtheorem{thm}{Theorem}[section]
\newtheorem{lem}[thm]{Lemma}
 
\newtheorem{cor}[thm]{Corollary}

\newtheorem{rem}[thm]{Remark}
\newtheorem{ex}[thm]{Example}

\newtheorem{Defn}[thm]{Definition}

\newcommand{\Ap}{A^{\phi}}
\newcommand{\Abar}{\overline{A}}

\newcommand{\phibar}{\overline{\phi}}

\renewenvironment{abstract}
{\begin{quote}
		\noindent \rule{\linewidth}{.5pt}\par{\bfseries \abstractname.}}
	{\medskip\noindent \rule{\linewidth}{.5pt}
	\end{quote}
}

\title{ 
 On triviality of $\mathbb{A}^{2}$-forms admitting a nontrivial $\mathbb{G}_a$-action
}
\author{
	Debojyoti Saha\\
	{\small{\it  Theoretical Statistics and Mathematics  Unit, Indian Statistical Institute,}}\\ 
	{\small{\it 203 B.T.Road, Kolkata-700108, India}}\\
	{\small{\it e-mail : debojyotisaha19@gmail.com}}}

\begin{document}
	
\date{}	
\maketitle


\begin{abstract}
	  In this paper, we give a structure theorem for $\mathbb{A}^2$-forms over arbitrary field extensions admitting a nontrivial $\mathbb{G}_a$-action.  As a consequence of the structure theorem, we prove that any $\mathbb{A}^2$-form is cancellative, a generalization of the Zariski Cancellation Theorem for the affine plane over an arbitrary field.  From the structure theorem we also derive some conditions under which an $\mathbb{A}^2$-form becomes trivial. In particular, we prove that over a field $k$, a factorial $\mathbb{A}^2$-form having a $k$-rational point and a non-trivial $\mathbb{G}_a$-action is trivial and we also give examples demonstrating that none of these hypotheses can be discarded.
	  
\end{abstract}	 

 	\smallskip
 \noindent
 {\small {{\bf Keywords}. }}
 \smallskip
 $\mathbb{A}^2$-form, Exponential Map ($\mathbb{G}_a$-action), Makar-Limanov Invariant, Cancellation.
 \smallskip
 \newline
 {\small{{\bf 2020 MSC}. Primary: 14R10, 13A50; 
 		Secondary: 13B25.

\section{Introduction}

All rings in this article are assumed to be commutative and unital. For an algebra $B$ over a ring $R$, we write ``$B=R^{[n]}$" for some $n\geq 1$ if $B=R[t_1,t_2,...,t_n]$ for some $t_1,t_2,...,t_n$ in $B$  algebraically independent over $R$, i.e, if $B$ is isomorphic to a polynomial ring in $n$ indeterminates over $R$. 
\indent

Let $k$ be any field.
 A $k$-algebra $A$ is said to be an  $\mathbb{A}^n$-form with respect to a field extension $L\mid _k$ if $A\otimes_{k}L=L^{[n]}$. An $\mathbb{A}^n$-form $A$ is said to be trivial if $A=k^{[n]}$.
 It is well-known that every separable $\mathbb{A}^2$-form over $k$ is trivial \cite[Theorem 3]{kambayashi}.
 However, for inseparable field extensions, there exist examples of nontrivial $\mathbb{A}^n$-forms even for $n=1$ (\cite{aasnuma_forms}, \cite{kam_miya}). To determine structure of $\mathbb{A}^n$- forms remains an active area of research. The following table summerizes the current status of research in this direction:  

  \indent

   \begin{table}[H]
   	\centering
   	\begin{tabular}{|c|c|c|c|}
   		\hline
   		 $L|_k$ & $n$ & Structure &\\
   		\hline
   		Separable & $1$ & { Trivial} & \\
   		\hline
   		Separable & $2$ & { Trivial} & T. Kambayashi, 1975 \\
   		\hline
   		Separable & $\geq 3$ & {Open } &\\
   		\hline 
   		Inseparable & $1$ & {Known if characteristic of $k\neq 2$} & T. Asanuma, 2005\\
   		\hline 
   		Inseparable & $\geq 2$ & {Open} &\\
   		\hline
   	\end{tabular}
   \end{table}

 In this paper we obtain a structure of $\mathbb{A}^2$-forms that admit a nontrivial $\mathbb{G}_a$-action. To best of author's knowledge this is first such result for inseparable $\mathbb{A}^n$-forms for $n=2$. Further once this result is obtained, it follows immediately from results of L. Makar-Limanov and P. Russell that any $\mathbb{A}^2$-form is cancellative. We also deduce a set of conditions under which an $\mathbb{A}^2$-form becomes trivial and give three examples demonstrating that these conditions are minimal.
 
 \indent
 
  {\it A $k$-algebra admitting nontrivial $\mathbb{G}_a$-actions is called   non-rigid, otherwise it is called rigid}. In particular $k^{[2]}$ is non-rigid. 
We establish the following structure theorem for nonrigid $A^2$-forms (Theorem \ref{strctr_thm_nr_forms}).

\indent

{\bf Theorem A.} Let $A$ be a non-rigid $\mathbb{A}^2$-form with respect to a field extension $L|_k$. Then $A=B^{[1]}$, where $B$ is an $\mathbb{A}^1$-form.

\indent

Recall that a $k$-algebra $A$ is said to be cancellative if for any $k$-algebra $B$, $A^{[1]}\cong_{k}B^{[1]}\Rightarrow A\cong_{k}B$. The Zariski Cancellation Theorem for affine plane states that when $A=k^{[2]}$, $A$ is cancellative. To be more precise, we have the following:

\begin{thm}[Zariski Cancellation Theorem]\label{zct}	
	Let $k$ be any field. Suppose $B$ is any $k$-algebra such that $B^{[1]}\cong_{k} k^{[3]}$. Then $B\cong_{k}k^{[2]}$.
\end{thm}

Theorem \ref{zct} was established by Fujita, Miyanishi and Sugie (\cite{fujita}, \cite{miyanishi_sugie}) in characteristic 0, by Russell \cite{russell} over perfect fields and later by Bhatwadekar-Gupta \cite{bhatwadekar_gupta} over arbitrary fields. As an application of Theorem A, we prove the following generalization (Theorem \ref{callelation of A^2-forms}).

\indent

\noindent
{\bf Theorem B.} Let $A$ be an $\mathbb{A}^2$-form over a field $k$ with respect to a field extension $L|_k$. Let $B$ be a $k$-algebra such that $A^{[1]}\cong_{k}B^{[1]}$. Then $A\cong_{k}B$.

\indent

   The following result on $\mathbb{A}^1$-forms follows from a theorem of T. Asanuma \cite[Theorem 4.8]{aasnuma_forms}. A self contained proof can be found in \cite[Corollary 4]{pdas}, where P. Das proves a more general result quoted in Remark \ref{remark P.das} below.

\begin{thm}\label{asanuma_u.f.d_retract}
	Let $k$ be any field. Let $A$ be an $\mathbb{A}^1$-form with respect to a field extension $L|_k$. Suppose the following hold:
	
	\begin{enumerate}
		\smallskip
		\item [{\rm (i)}] $A$ is a U.F.D
		
		\item [{\rm (ii)}] $A$ has a $k$-rational point.
	\end{enumerate}
	Then $A=k^{[1]}$.
	
\end{thm}
As an application of Theorem A we get the following analogue of Theorem \ref{asanuma_u.f.d_retract} for $\mathbb{A}^2$-forms  (Theorem \ref{mtheo1}).


\indent

\noindent
{\bf Theorem C.} Let $A$ be a non-rigid factorial $\mathbb{A}^{2}$-form over a field $k$ with respect to a field extension $L\mid_k$ having a $k$-rational point. Then $A=k^{[2]}$.

\indent

  We also give examples demonstrating that none of the assumptions in Theorem B can be discarded; in particular, the hypotheses in the theorem are minimal. 

\indent

  





 
 {\bf Acknowledgment.} The author expresses his heartfelt thanks to professor Neena Gupta for her valuable suggestions.
  
  \section{Preliminaries}
 Let $A$ be an affine $k$-domain. We first recall the definition of an exponential map, the equivalent notion of a $\mathbb{G}_a$-action, on $A$. 
  
  \begin{Defn}
  		\rm{Let $\phi: A\longrightarrow A^{[1]}$ be a $k$-algebra homomorphism. For any indeterminate $U$ over $A$, let  us denote the map $\phi: A\longrightarrow A[U]$ by the notation $\phi_U$. The map $\phi$ is said to be an exponential map on $A$ if:}
  	
  	\begin{enumerate}
  		\item[\rm(i)] $\epsilon_0\circ\phi_U $ is identity on $A$, where $\epsilon_0: A[U]\longrightarrow A$  is the evaluation map at $U=0$.
  		
  		\item[\rm(ii)] $\phi_V\circ\phi_U=\phi_{V+U}$, where $\phi_V: A[U]\longrightarrow A[U,V]$ is the extension of the  homomorphism $\phi_V: A\longrightarrow A[U]$ by $\phi_V(U)=U$.
  	\end{enumerate} 
  \end{Defn}
   The subring $A^{\phi}:=\{a\in A: \phi(a)=a\}$ is called the ring of invariants of the map $\phi$. The exponential map $\phi$ is called nontrivial if $A\neq A^{\phi}$.
   We denote the set of all exponential maps on $A$ by ${exp}(A)$. The Makar-Limanov invariant of $A$ denoted by ${\rm ML}(A)$ is defined as \[{\rm ML}(A)=\bigcap_{\phi\in exp(A)} A^{\phi}.\] 
   
   \indent
    Given an exponential map $\phi:A\longrightarrow A[U]$, define, for each $n\geq 0$ , $k$-linear functions $\phi^n: A\longrightarrow A$ by $\phi^n(a)= $ the coefficient of $U^n$ in the polynomial $\phi(a)$. For each $a\in A$ denote 
  ${\rm deg}_U(\phi(a))$ by ${\rm deg}_{\phi}(a)$. It follows that  
    
    	\begin{enumerate}
    		\item[\rm(i)]  ${\rm deg}_{\phi}(a)=-\infty \Leftrightarrow a=0.$
    		\item[\rm(ii)] ${\rm deg}_{\phi}(ab)={\rm deg}_{\phi}(a)+{\rm deg}_{\phi}(b).$
    		\item[\rm(iii)]  ${\rm deg}_{\phi}(a+b)\leq max\{{\rm deg}_{\phi}(a),{\rm deg}_{\phi}(b)\}$.
    	\end{enumerate}
   We note that $A^{\phi}=\{a\in A: {\rm deg}_{\phi}(a)=0\}$. Thus it follows that $ A^{\phi}$ is inert in $A$, i.e, $ab\in A^{\phi}\implies a, b \in A^{\phi} $ for all non-zero $a, b\in A$. In particular, $ A^{\phi}$ is algebraically closed in $A$ and if $A$ is a U.F.D then $ A^{\phi}$ is also a U.F.D.

   	
   	
  
   
   
  We record the following lemma which can be found in \cite[Lemma 2.2 (c) and (e)]{crachiola_makar_limanov}.
  
  \begin{lem}\label{local_slice_thm}
  	Let $\phi$ be any non-trivial exponential map on an affine $k$-domain $A$. Then there exists another exponential map $\psi:A\longrightarrow A[U]$ on $A$ such that $A^{\psi}=A^{\phi}$ and the set $S:=\{a\in A| {\text deg}_{\psi}(a)=1\}$ is non-empty. Further, for any $x\in S$ with $\psi(x)=x+\pi U$, we have $A_{\pi}=A^{\psi}_{\pi}[x]\cong_{k}(A^{\psi}_{\pi})^{[1]}.$ 
  \end{lem}
  
    The following rigidity theorem is due to Crachiola and Makar-Limanov \cite[Theorem 3.1]{crachiola_makar_limanov}.
  
  \begin{thm}\label{rigidity_thm}
  	Let $k$ be any field. Suppose $A$ is an affine $k$-domain with ${\rm ML}(A)=A$. Then ${\rm ML}(A^{[1]})=A$.
  \end{thm}

   The following result is well-known (see \cite[2.6]{AEH}).
   \begin{thm}\label{AEH}
   	Let $k$ be any field. Then any one dimensional normal $k$-subalgebra $A$ of a polynomial ring $k^{[n]}$ is $k^{[1]}$. 
   \end{thm}
   
   The following strong invariance property of one-dimensional affine domains is due to Abhyankar, Eakin and Heinzer \cite[Theorem 3.3]{AEH}. 
   \begin{thm}\label{strong_invariance}
   	Let $k$ be any field. Let $A$, $B$ be k-domains of transcendence degree 1 over $k$ such that \[\phi:A[t_1,t_2,..,t_n]\longrightarrow B[z_1,z_2,...,z_n]\] is an isomorphism for some $n\geq 1$, where $t_1,...,t_n$ are algebraically independent over $A$ and $z_1,...,z_n$ are algebraically independent over $B$. Then $\phi(A)=B$ or $A\cong_{k}B\cong_{k} k^{[1]}$.
   \end{thm}

      The following lemma follows from \cite[Theorem 3.1]{crachiola_fac} and Theorem \ref{AEH}.
   \begin{lem}\label{ker_exp_plane}
   	For an algebraically closed field $k$, the ring of invariants of any nontrivial exponential map $\phi$ on $k[X,Y](=k^{[2]})$ is $k[f]=k^{[1]}$ for some $f\in k[X,Y]$ such that $k[X,Y]=k[f]^{[1]}$.
   \end{lem}
   
   
   
   
   	
   		
   	

\section{Main Results}

		We first prove Theorem A.
		\begin{thm}\label{strctr_thm_nr_forms}
			Let $k$ be a field and $L$ a field extension of $k$ with algebraic closure $\overline{L}$.
			Let $A$ be a non-rigid $\mathbb{A}^2$-form with respect to $L|_k$. Then $A=B^{[1]}$, where $B$ is an $\mathbb{A}^1$-form with respect to $\overline{L}|_k$.
		\end{thm}
		\begin{proof}
			
			\indent
			Let $\phi:A\longrightarrow A[U]$ be a nontrivial exponential map on $A$ and $B=A^{\phi}$. From Lemma \ref{local_slice_thm}, we may assume that $\phi$ admits a local slice, i.e, the set $S:=\{a\in A| {\rm deg}_{\phi}(a)=1\}$ is non-empty. Replacing $L$ by  $\overline{L}$, we may assume that $L$ is algebraically closed. 
			Set \[\overline{A}:=A\otimes_kL=L^{[2]}  \text{  and  } \overline{B}:=B\otimes_kL. \] We can extend $\phi$ to an exponential map
			
			 \[\phibar=\phi\otimes_{k}{\rm id}_{L} : \Abar \longrightarrow \Abar[U] \;\;{\text{	with $\overline{A}^{\phibar}= \overline{B}$ }}.\]  By Lemma \ref{ker_exp_plane}, we get $X\in \overline{B}$ and $Y\in \overline{A}$ such that $\overline{B}=L[X]$ and $\Abar= L[X][Y]=L^{[2]}$. Let $F\in S$ be such that the $X$-degree of $\phi^1(F) $ is the least among the elements of $\{\phi^1(a)|a\in S\}\subset B$. Let $\phi(F)=F+\pi U$ for some $\pi\in B$. 
			
			From Lemma \ref{local_slice_thm}, we have 
				\begin{equation}
				\overline{B}_{\pi}[F]=\overline{A}_{\pi}.
			\end{equation}
		
			Also from the relation  $\Abar= L[X][Y]$, we get
			\begin{equation}
				\overline{B}_{\pi}[Y ]=\overline{A}_{\pi}.
			\end{equation} 
		Thus, it follows from equations (1) and (2) that $F=\lambda Y+\mu$ where $\lambda\in (\overline{B}_{\pi})^*$ and $\mu \in \overline{B}_{\pi}$.
			As $F\in L[X,Y]$ we have $\lambda, \mu \in \overline{B}=L[X]$. Let $\phibar^1(Y)=g\in L[X]$.
			
			 
			  Applying $\phibar$ on both sides of the relation $F=\lambda Y+\mu$ we get 
			\begin{equation}
				\pi=\lambda g.
			\end{equation}
			
			As $Y\in \overline{A}$, there exist $k$-linearly independent elements $l_1,l_2,...,l_n\in L$ such that
			\begin{equation}
				Y=a_1\otimes_{k}l_1+a_2\otimes_{k}l_2+...+a_n\otimes_{k}l_n 
			\end{equation}
			for some $a_1,a_2,...,a_n\in A$. Comparing the $\phibar$-degrees on both sides of the equation (4) and using the fact that $l_1,l_2,...,l_n\in L$ are $k$-linearly independent, we get ${\rm deg}_{\phi}(a_i)\leq 1$ for all $i, 1\leq i\leq n$ and ${\rm deg}_{\phi}{a_i}=1$ for at least one $i, 1\leq i\leq n$. Thus applying $\phibar$ on both sides of equation (4), we get
			\begin{equation}
				g= \phi^1(a_1)\otimes_{k}l_1+\phi^1(a_2)\otimes_{k}l_2+...+\phi^1(a_n)\otimes_{k}l_n.
			\end{equation}
			As $l_1,l_2,...,l_n\in L$ are $k$-linearly independent, from equation (5), we get $${\rm deg}_X(g)={\rm max}\{{\rm deg}_X(a_i)|1\leq i\leq n\}\geq {\rm deg}_X(\pi)\text{ from the choice of $F$}.$$
			 Hence, from equation (3), we get $\lambda\in L^*$. Thus, $B[F]\otimes_kL=A\otimes_kL$ and as $L$ is faithfully flat over $k$, we get $A=B[F]$. This completes the proof.
		 
		   
		   
		   
		
     	
	   
	   
		\end{proof}

	\begin{cor}\label{form_inside_poly_ring}
		Let $k$ be a field and $A$ be a non-rigid $\mathbb{A}^2$-form with respect to some field extension $L|_k$. If $A\subset k^{[n]}$ for some $n\geq 2$, then $A=k^{[2]}$. 
	\end{cor}	
	\begin{proof}
	 By Theorem \ref{strctr_thm_nr_forms}, $A=B^{[1]}$, where $B$ is an $\mathbb{A}^1$-form with respect to $\overline{L}|_k$. Now as $B\subseteq A\subseteq k^{[n]}$, by Theorem \ref{AEH}, we get $B=k^{[1]}$. Hence $A=k^{[2]}$. 
	\end{proof}

	\begin{cor}\label{ML_trivial_form_trivial}
		Let $A$ be an $\mathbb{A}^2$-form over $k$. Then the following are equivalent:
		\begin{enumerate}
			\item[\rm(i)] {\rm ML}$(A)$=$k$.
			\item[\rm(ii)] $A=k^{[2]}$.
		\end{enumerate}
	\end{cor}
	\begin{proof}
		$(ii)\Rightarrow(i):$ This is trivial.
		\indent
		
		$(i)\Rightarrow(ii):$
		As ${\rm ML}(A)=k$, $A$ is nonrigid. By Theorem \ref{strctr_thm_nr_forms}, $A=B^{[1]}$ for some k-subalgebra $B$ of $A$. If $B$ were rigid, then, by Theorem \ref{rigidity_thm}, ${\rm ML}(A)=B\neq k$. Thus $B$ is also nonrigid and as $k={\rm ML}(A)$ is inert in $A$, it is algebraically closed in $A$, it follows from \cite[Lemma 2.2 (d)]{crachiola_makar_limanov} that $B=k^{[1]}$. Hence $A=k^{[2]}$. 
	\end{proof}
	
	\begin{cor}\label{struc_using_asanuma_and_mine}
		Let $k$ be a field of characteristic $p>2$. Let $A$ be a non-rigid $\mathbb{A}^2$-form with respect to $\overline{k}|_k$, where $\overline{k}$ is the algebraic closure of $k$. Then there exist indeterminates $X,Y$ over $L$ such that $A\otimes_kL=L[X,Y]$ and
		\[A=k[X^{p^e},\alpha^2\beta, \alpha\beta^{\lambda}][Y]\] for some $e\geq0$ and $\alpha,\beta \in k^{1/p^e}[X]$ and $\lambda=(p^e+1)/2$.
	\end{cor}
	
	\begin{proof}
		Follows from Theorem \ref{strctr_thm_nr_forms} and \cite[Theorem 8.1]{aasnuma_forms}.
	\end{proof}

		\indent 
	As an application of Theorem \ref{strctr_thm_nr_forms}, we obtain Theorem B which proves that $\mathbb{A}^2$-forms are cancellelative. We note the result gives an alternative proof of Theorem \ref{zct} in \cite{bhatwadekar_gupta}.  
	
	\begin{thm}\label{callelation of A^2-forms}
		Let $A$ be an  $\mathbb{A}^2$-form over a field $k$. Let $B$ be a $k$-algebra such that $A^{[1]}\cong_{k}B^{[1]}$. Then $A\cong_{k}B$.
	\end{thm} 
	
	\begin{proof}
		Let $\varphi: A[X] \longrightarrow B[Y]$ be the given isomorphism, where $X,Y$ are indeterminates over $A$ and $B$ respectively. 
		Identifying $A[X]$ with its image in $B[Y]$, let us assume that $A[X]=B[Y]$. Without loss of generality we may assume $A$ is an $\mathbb{A}^2$-form with respect to $L|_k$ where $L$ is an algebraically closed field. As $L^{[2]}$ is cancellative (\cite{russell}), it follows that $B$ is also an $\mathbb{A}^2$-form. Suppose that $A$ is a rigid domain. Then, by Theorem \ref{rigidity_thm},
		
		\[{\rm ML}(A[X])=A.\] 
		Hence, \[A={\rm ML}(B[Y])\subseteq B.\] Thus, as ${\rm trdeg}_k(A)={\rm trdeg}_k(B)=2$ and $A$ is algebraically closed in $B[Y]$, we get $A=B$. 
		
		Now suppose that $A$ is nonrigid. Then, by the argument as before, $B$ is also nonrigid. Thus, by Theorem \ref{strctr_thm_nr_forms}, $A=R[U]=R^{[1]}$ and $B=S[V]=S^{[1]}$ for some $\mathbb{A}^1$-forms $R$ and $S$ over $k$.
		As 
		\[R[U,X]=S[V,Y],\]
		by Theorem \ref{strong_invariance}, it follows that \[R\cong_{k}S.\]
		Hence $A\cong_{k}B$.

\end{proof}

		As another application of Theorem \ref{strctr_thm_nr_forms}, we have Theorem C below.
		\begin{thm}\label{mtheo1} \label{factorial_retract_nrigid_trivial}
			Let $A$ be an $\mathbb{A}^{2}$-form over a field $k$ with respect to any field extension $L\mid_k$. Suppose the following hold
			
				\begin{enumerate}
				\smallskip
				\item[{\rm(i)}] $A$ is a U.F.D
				\smallskip
				\item[{\rm(ii)}] $A$ has a $k$-rational point (or equivalently $A$ has a retraction to $k$) and
				\smallskip
				\item[{\rm(iii)}] $A$ is non-rigid.
				
			\end{enumerate}
			 Then $A=k^{[2]}$.
		\end{thm}
		\begin{proof}
			By Theorem \ref{strctr_thm_nr_forms}, we have that $A=B^{[1]}$ for some $\mathbb{A}^1$-form $B$. The rest follows from Theorem \ref{asanuma_u.f.d_retract}.
		\end{proof}

		The following three examples show that none of the hypotheses in Theorem \ref*{mtheo1} can be discarded.
		 
		\begin{ex}\label{ex 2}
		\rm Let $\mathbb{F}_p$ be the finite field of characteristic $p>0$. Consider the field $k=\mathbb{F}_p(t^p,u^p)$ where $t$ and $u$ are indeterminates over $\mathbb{F}_p$. For indeterminates $X$ and $Y$ over $k$, let us consider $A=k[X^p,t+X+uX^p,Y]$. As $k^{[1]}$ is cancellative (\cite[Corollary 2.8]{AEH}) and $k[X^p,t+X+uX^p]\ncong_k k^{[1]}$ is a U.F.D without a retraction to $k$ \cite[p. 70-71 Remark 6.6(a) example (i)]{kam_miya}, we get $A$ is a nontrivial $\mathbb{A}^2$-form which is a non-rigid U.F.D without a retraction to $k$.  
		
		\end{ex}
	 
		\begin{ex}\label{ex 3}
			\rm
			Let $k$ be a field of characteristic $p\geq 2$ and $A=\frac{k[X,Y,Z]}{(Y^p-X-aX^p)}$, where $a\in k\setminus k^p$. As $k^{[1]}$ is cancellative (\cite[Corollary 2.8]{AEH}) and $\frac{k[X,Y]}{(Y^p-X-aX^p)}\ncong_kk^{[1]}$ is a non-U.F.D with a retraction to $k$ \cite[p. 70-71 Remark 6.6(a) example (ii)]{kam_miya}, we get that $A$ is a non-trivial, nonrigid $\mathbb{A}^{2}$-form over $k$ with a retraction to $k$ but $A$ is not a U.F.D.
		\end{ex} 
		
		The next example shows that the assumption on non-rigidity cannot be discarded.
		\begin{ex}\label{ex 4}
			\rm Let $\mathbb{F}_2$ be the finite field of characteristic $2$ and $\alpha$, $\beta$ be two indeterminate over $\mathbb{F}_2$. Let $L= \mathbb{F}_2(\alpha,\beta)$.  Consider the subfield $k=\mathbb{F}_2({\alpha}^2,{\beta}^2)$ of $L$. Consider the affine $k$-domain \[A=
			\frac{k[U,V,W]}{(W^2+U+{\alpha}^2U^2+{\beta}^2V^2)}.\]
			Let $u,v,w$ respectively denote the images of $U,V,W$ in $A$. We observe that in $L[U,V,W]$, the polynomial $W^2+U+{\alpha}^2U^2+{\beta}^2V^2$ can be written as $(W+\alpha U)^2+U+{\beta}^2V^2$ which is a coordinate in $L[U,V,W]$. Thus $A\otimes_kL=L^{[2]}$. Hence $A$ is an $\mathbb{A}^2$-form over $k$ having a retraction to $k$. Now we show that $A$ is a U.F.D. As $u$ is a prime element of $A$, using Nagata'a criterion for U.F.D, it's enough to prove that $A_u(:= A[1/u])$  is a U.F.D. We have,
			\[\begin{split}
				A_u&=\frac{k[U,1/U,V,W]}{((W/U)^2+1/U+{\alpha}^2+{\beta}^2(V/U)^2)}\\
				&\cong_{k}\frac{k[X,Y,Z,1/X]}{(Z^2+X+{\alpha}^2+{\beta}^2Y^2)}.\\
			\end{split}
			\]
			Thus $A_u$ is a localization of $k^{[2]}$, and hence a U.F.D implying that $A$ is a U.F.D. 
			
			Since $A$ is a normal domain and $\alpha,\beta$ is integral over $A$ such that $\alpha,\beta,\alpha/\beta\notin A$, it follows that if $g\in A$ then $\alpha g,\beta g,(\alpha/\beta) g\notin A$.  
			Now we prove that $A$ is a rigid ring and thus is not $k^{[2]}$.
			\newline
			\indent
			If possible, suppose $\phi$ is a non-trivial exponential map on $A$. Set $w_1:=w+v$, $l:={\rm deg}_{\phi}(w_1)$, $n:={\rm deg}_{\phi}(u)$ and $m:={\rm deg}_{\phi}(v)$. We first prove that $w_1 $ cannot be in $\Ap$ i.e $l\neq 0$.
			
			\indent
			 Let if possible $l=0$. Then none of $u$ and $v$ is in $\Ap$ as $\phi$ is nontrivial. Thus $n>0$ and $m>0$. In this case from the relation 
			\[w_1^2+u+{\alpha}^2u^2+v^2+{\beta}^2v^2=0\]
			 we get 
			 \[{\alpha}^2(\phi^n(u))^2=({\beta}^2+1)(\phi^m(v))^2.\]
			 Thus \[\phi^n(u)=\left(\frac{1+\beta}{\alpha}\right)(\phi^m(v)),\]
			 which is a contradiction as $\frac{(1+\beta)}{\alpha}$ is not in $k$. Hence $w_1$ is not in $\Ap$ i.e $l\geq1$.
			 
			  Now if $n=0$, then from the relation
			  \[\phi(w_1^2+u+{\alpha}^2u^2+v^2+{\beta}^2v^2)=0,\] 
			 we get
			  \[(\phi^l(w_1))^2= (1+{\beta}^2)(\phi^m(v))^2.\]
			  Thus
			   \[\phi^l(w_1)=(1+\beta)(\phi^m(v)) ,\]
			   which is a contradiction as $\beta$ is not in $k$.
			   Therefore $n\geq 1$. Then again from the relation
			   \[\phi(w_1^2+u+{\alpha}^2u^2+v^2+{\beta}^2v^2)=0,\] we get
			    
			 \[{\alpha}^2(\phi^n(u))^2= (\phi^n(w_1))^2+(1+{\beta}^2)(\phi^n(v))^2.\]
			And hence 
			 \[\phi^n(w_1)={\alpha}(\phi^n(u))+(1+\beta)\phi^n(v),\] 
			 which is a contradiction as $\alpha/\beta$ is not in $k$. Thus $A$ is a rigid factorial $\mathbb{A}^2$-form over $k$ having a retraction to $k$.
			    
		\end{ex}
		The following example shows that Theorem \ref{mtheo1} cannot be generalized for $\mathbb{A}^3$-forms.
		
		\begin{ex}\label{ex 5}
			
		\rm	Let $k$ be any non-perfect field of characteristic $p>0$. Let $\lambda$ be an element in $ k\setminus k^p$ and set $L=k(\lambda)$. Consider the affine $k$-domain 
			\[A=\frac{k[X,Y,Z,T]}{(X^mY-Z^p-\lambda T^p-T)} \text{ for some $m\geq 2$}.\] Let $x,y,z,t$ denote the images of $X,Y,Z,T$ respectively in $A$.
			Then $A$ admits a non-trivial exponential map $\phi: A\rightarrow A[U]$ given by 
			
			\begin{alignat*}{3}
			&\phi(x)&=&x\\
			&\phi(t)&=&t\\
			&\phi(z)&=&z+x^mU\\
			&\phi(y)&=&\frac{(z+x^mU)^p+\lambda t^p+t}{x^m} =y+x^{mp-m} U^p.
		\end{alignat*}
			
			 As $Z^p+ \lambda T^p+ T$ is a coordinate in $L[Z,T]$ but not in $k[Z,T]$, by \cite[Theorem 3.11]{ng_comp}, $\overline{A}:=A\otimes_kL =L[x]^{[2]}=L^{[3]}$ but $A\neq k[x]^{[2]}$. Further, by \cite[Lemma 3.1]{ng_comp}, $A$ is a U.F.D and clearly $A$ is a non-rigid $\mathbb{A}^3$-form having a retraction to $k$ but $A\neq k^{[3]}$.   
			
		\end{ex}

		\begin{rem}\label{remark P.das}
		{\rm 	
			Let $k$ be a field. Let $R$ be a $k$-algebra and $A$ an $R$-algebra such that $A$ is a U.F.D with a retraction to $R$. In \cite{pdas}, P. Das had shown that if $A$ is an $\mathbb{A}^1$-form over $R$ with respect to $L|_k$, i.e, if $A\otimes_kL= (R\otimes_kL)^{[1]}$ for some field extension $L|_k$, then $A=R^{[1]}$. Example \ref{ex 5} also shows that even when $R$ is a P.I.D like $k^{[1]}$, an analogous result for Theorem \ref{factorial_retract_nrigid_trivial} does not hold for $\mathbb{A}^2$-forms over $R$, i.e, Theorem \ref{factorial_retract_nrigid_trivial} does not extend to $\mathbb{A}^2$-forms over P.I.D's.
			
		}
			
		\end{rem}
		
		\noindent
		{\bf Data Availability} { \scriptsize Not applicable as no data was used or analyzed in this work.}
	\newline
	
	\noindent
		{\large \bf Declarations}
	\newline
		
		\noindent
	 {\bf Conflict of Interest} {\scriptsize The author declares no competing interests.}

\end{document}